\documentclass[12pt,a4paper]{amsart}
\usepackage{latexsym,amsfonts,amsthm,amsmath,mathrsfs,amssymb}
\usepackage[british]{babel}
\usepackage[dvips]{color}
\usepackage[utf8]{inputenc}
\usepackage[T1]{fontenc}
\usepackage[dvips,final]{graphics}
\usepackage{xcolor}
\usepackage{mathrsfs}
\usepackage{mathtools} 
\usepackage{breqn}
\usepackage{epstopdf}
\usepackage{verbatim}
\usepackage{amscd}
\usepackage{hyperref}
\usepackage[english,capitalize]{cleveref}
\usepackage{aliascnt}

\usepackage{todonotes}

\makeatletter
\def\newaliasedtheorem#1[#2]#3{
	\newaliascnt{#1@alt}{#2}
	\newtheorem{#1}[#1@alt]{#3}
	\expandafter\newcommand\csname #1@altname\endcsname{#3}
}
\makeatother

\theoremstyle{plain}

\newtheorem{theorem}{Theorem}[section]
\newaliasedtheorem{prop}[theorem]{Proposition}
\newaliasedtheorem{lem}[theorem]{Lemma}
\newaliasedtheorem{coroll}[theorem]{Corollary}

\theoremstyle{definition}
\newaliasedtheorem{defi}[theorem]{Definition}
\newaliasedtheorem{quest}[theorem]{Question}
\newaliasedtheorem{fact}[theorem]{Fact}

\theoremstyle{remark}
\newaliasedtheorem{rem}[theorem]{Remark}
\newaliasedtheorem{exa}[theorem]{Example}

\newcommand{\R}{\mathbb{R}}

\newcommand{\C}{\mathbb{C}}

\newcommand{\HH}{\mathbb{H}}

\newcommand{\G}{\mathbb{G}}

\let\altphi\phi
\let\phi\varphi
\let\varphi\altphi
\let\altphi\undefined

\newcommand{\average}{{\mathchoice {\kern1ex\vcenter{\hrule height.4pt
width 6pt
depth0pt} \kern-9.7pt} {\kern1ex\vcenter{\hrule height.4pt width 4.3pt
depth0pt}
\kern-7pt} {} {} }}

\address{\textsc{Daniela Di Donato}: 
Dipartimento di Ingegneria Industriale e Scienze Matematiche, Via Brecce Bianche, 12 60131 Ancona, Universit\'a Politecnica delle Marche.}
\email{d.didonato@staff.univpm.it}

\title{Intrinsic Lipschitz sections of no-linear quotient maps}

\date{\today}

\author{ Daniela Di Donato}

\begin{document}

\begin{abstract}
		 Le Donne and the author introduced the so-called intrinsically Lipschitz sections of a fixed quotient map $\pi$ in the context of metric spaces. Moreover, the author introduced the concept of intrinsic Cheeger energy when the quotient map is also linear. In this note we investigate about the no-linearity of $\pi$. In particular, we find a Leibniz  formula for the intrinsic slope when $\pi$ satisfies a weaker condition. After that, we focus our attention on Carnot groups and using the properties of intrinsic dilations we show that the dilation of a Lipschitz section is so too. Finally, in Carnot groups of step 2, we give a suitable additional condition in order to get the sum of two intrinsically Lipschitz sections is so too. 
\end{abstract}

\maketitle 
\tableofcontents

\section{Introduction}
Find out a good notion of rectifiability in subRiemannian Carnot groups \cite{ABB, BLU, CDPT} is unanswered questions which have continued to perplex many mathematicians for decades. Starting from a negative result by Ambrosio and Kirchheim \cite{AmbrosioKirchheimRect}, Franchi, Serapioni and Serra Cassano introduced the notion of intrinsically Lipschitz maps in  \cite{FSSC, FSSC03, MR2032504}  (see also  \cite{SC16, FS16}) in order to answer this question. Recently, Le Donne and the author  \cite{DDLD21} generalize this concept in metric spaces and prove some relevant properties like Ahlfors-David regularity, Ascoli-Arzel\'a Theorem, Extension theorem, etc. for the so-called intrinsically Lipschitz sections.

Our setting is the following. We have a metric space $X$, a topological space $Y$, and a 
quotient map $\pi:X\to Y$, meaning
continuous, open, and surjective.
The standard example for us is when $X$ is a metric Lie group $G$ (meaning that the Lie group $G$ is equipped with a left-invariant distance that induces the manifold topology), for example a subRiemannian Carnot group, 
and $Y$ if the space of left cosets $G/H$, where 
$H<G$ is a  closed subgroup and $\pi:G\to G/H$ is the projection modulo $H$, $g\mapsto gH$.
\begin{defi}\label{def_ILS} 
We say that a map $\phi :Y \to X$ is a section of $\pi $ if 
\begin{equation}\label{equation1}
\pi \circ \phi =\mbox{id}_Y.
\end{equation}
Moreover, we say that a map $\phi:Y\to X$ is an {\em intrinsically Lipschitz section of $\pi$ with constant $L$},  with $L\in[1,\infty)$, if in addition
\begin{equation}\label{equationFINITA}
d(\phi (y_1), \phi (y_2)) \leq L d(\phi (y_1), \pi ^{-1} (y_2)), \quad \mbox{for all } y_1, y_2 \in Y.
\end{equation}
Here $d$ denotes the distance on $X$, and, as usual, for a subset $A\subset X$ and a point $x\in X$, we have
$d(x,A):=\inf\{d(x,a):a\in A\}$.
\end{defi}
A first observation is that this class is contained in the class of continuous maps (see \cite[Section 2.4]{DDLD21}) and in the case  $ \pi$ is a Lipschitz quotient or submetry \cite{MR1736929, Berestovski},  being intrinsically Lipschitz  is equivalent to biLipschitz embedding, see Proposition 2.4 in \cite{DDLD21}. Moreover, since $\phi$ is injective by \eqref{equation1}, the class of Lipschitz sections not include the constant maps. 

Conversely to Lipschitz classical notion, the intrinsically Lipschitz sections cannot be uniformly continuous; an easy example is as follows.
 \begin{exa}[$\pi$ is not linear]\label{exa3}
Let $X=(0,1), Y=(1,+\infty)$ and  $\pi : (0,1) \to (1,+\infty)$ defined as $\pi (x)=\frac 1 x,$ for every $x\in (0,1).$ It is easy to see that $\pi$ is a quotient map and  $\phi : (1,+\infty) \to (0,1)$  defined as $$ \phi (y)= \frac 1 {y},\quad \forall y\in (1,+\infty) ,$$ is an intrinsic $1$-Lipschitz  section of $\pi$. On the other hand, it is well-known that $\phi$ is not uniformly continuous.
  \end{exa}

 The long-term objective is to obtain \textbf{a Rademacher type theorem \'a la Cheeger } \cite{C99} (see also  \cite{KM16, K04}) for the intrinsically Lipschitz sections.
 
  This project fits to an active line of research \cite{Pansu, Vittone20, FMS14} on Rademacher type theorems, but it also studies connections to different notions and mathematical areas like Cheeger energy, Sobolev spaces, Optimal transport theory. The reader can see \cite{AGS08.2, AGS08.3, S06, MR2480619, MR2459454, H96, S00, HKST15}.

A first step of this research is given in  \cite[Proposition 3.1-3.2]{D22.1} and \cite{D22.31may} as follows.
 \begin{prop}
Let $\pi: X\to Y$ be a linear and quotient map between a normed space $X$ and a topological space $Y.$ 
\begin{enumerate}
\item the set of all sections is a vector space over $\R$ or $\C.$
\item If $\phi: Y \to X$ is an intrinsically Lipschitz section of $\pi,$ then for any $\lambda \in \R-\{0\}$ the section $\lambda \phi$ is also intrinsic Lipschitz for $1/\lambda \pi$ with the same Lipschitz constant.
\item If $X=\R^s$, $Y$ is a metric space and each fiber is contained in a straight line and \eqref{equationpinonL} holds, then the sum of two intrinsically Lipschitz sections is so too. 
\end{enumerate}
\end{prop}
Here it is fundamental that $\pi$ is a linear map. The aim of this paper is to prove the last result when $\pi$ does not satisfy this property. We focus our attention on that property because it is not true in the context of Carnot groups. More precisely, the main results are 
\begin{enumerate}
\item Proposition \ref{sommadiLipeLip.2senzaa}: we give a Leibniz formula for the intrinsic slope when $\pi$ satisfies a weaker linearity condition.
\item Proposition \ref{carnotgroupsIDEA} and  Corollary \ref{coroll15g}: we focus our attention on the well-known properties of intrinsic dilations in Carnot groups  and we show that the dilation of a Lipschitz section is so too. 
\item Proposition \ref{carnotgroupsIDEA.12}:  \textbf{an open question is whether the sum of two Lipschitz maps in the Franchi, Serapioni and Serra Cassano sense is so too.}  Our idea of proof in  \cite{D22.31may} does not work to solve this open question in general case; on the other hand, if we just consider the linearity of $\pi$ valued in the points of the two sections, then it is possible to get a positive answer  in the context of Carnot groups of step 2 \cite{BLU, DD20}.

\end{enumerate}

\section{A weaker linearity condition for $\pi$} 

In this section, we study the properties proved in  \cite{D22.31may} when $\pi :X \to Y$ is not linear. Regarding the definition of the intrinsic Lipschitz constants the reader can see \cite[Section 3]{D22.31may}. Here, we ask that there exists $\lambda > 0$ such that
\begin{equation}\label{equationpinonL}
\pi (\alpha x_1+\beta x_2)= \alpha ^\lambda \pi(x_1) + \beta^\lambda \pi(x_2), \quad \mbox{ for any } x_1, x_2 \in X,
\end{equation}
and for any  $\alpha , \beta \in \R$ with $(\alpha , \beta) \ne (0,0).$

\begin{prop}\label{prop11giu.0}
Let $\phi , \psi :Y \to X$ be sections of $\pi.$ Then, for any $\alpha , \beta \in \R$ with $(a,b) \ne (0,0)$ the map $\eta :=\alpha \phi +\beta\psi :Y \to X$ is a section of $(\frac 1 {\alpha ^\lambda + \beta ^\lambda }) \pi.$
\end{prop}

\begin{proof}
Fix $\alpha , \beta \in \R$ with $(\alpha , \beta) \ne (0,0).$ The thesis follows because $\pi (\eta (y)) =\pi (\alpha \phi (y)+\beta \psi(y)) = (\alpha ^\lambda + \beta ^\lambda) y$ for any $y\in Y.$
\end{proof}

\begin{prop}\label{prop11giu.2}
Let $X$ be a normed space and $\phi :Y \to X$ be an intrinsically $L$-Lipschitz section of $\pi.$ Then, for any $\alpha \in \R - \{ 0 \}$  the map $\alpha \phi  :Y \to X$ is an intrinsically Lipschitz  section of $(\frac 1 {\alpha^\lambda  }) \pi$ with the same Lipschitz constant.
\end{prop}

\begin{proof}
Fix $\alpha \in \R - \{ 0 \}$ and $y, z \in Y.$  By Proposition \ref{prop11giu.0}, we have that $\alpha \phi$ is a section of $(\frac 1 {\alpha ^\lambda  }) \pi.$ On the other hand, noting that if $x \in  \pi^{-1} (z),$  then
\begin{equation}\label{equation11giu.1}
\alpha x \in  \pi ^{-1} (\alpha ^\lambda z) = \left(\frac 1 {\alpha ^\lambda  } \pi \right)^{-1} (z);
\end{equation}
indeed,
\begin{equation*}
\pi (\alpha x )= \alpha ^\lambda \pi(x) =\alpha ^\lambda z,
\end{equation*}
and so $\frac 1{ \alpha ^\lambda } \pi (\alpha x ) =z.$ Now we choose $x \in  \pi^{-1} (z)$ such that 
\begin{equation*}
d\left(\alpha \phi (y), \left(\frac 1 {\alpha ^\lambda  } \pi \right)^{-1} (z) \right)= d(\alpha \phi (y) , \alpha x).
\end{equation*}
Consequently,
\begin{equation*}
\begin{aligned}
d(\alpha \phi (y), \alpha \phi (z)) & = |\alpha| d(\phi (y), \phi (z)) \leq L |\alpha | d(\phi (y), \pi ^{-1} (z))\\
&  \leq L |\alpha | d(\phi (y), x) = L d\left(\alpha \phi (y), \left(\frac 1 {\alpha ^\lambda  } \pi \right)^{-1} (z) \right).
\end{aligned}
\end{equation*}

\end{proof}

 \begin{theorem}\label{theorem11.3}
 Let  $X=\R^s$, $Y$ be a metric space and $\pi: X \to Y$ be a quotient map such that each fiber is contained in a straight line. If $\pi$ satisfies \eqref{equationpinonL}, then the sum of two intrinsically Lipschitz sections is so too. 
\end{theorem}

\begin{proof}
The proof is exactly as  in \cite[Theorem 4.1]{D22.31may}, noting that the fiber of the sum is $(1/2\pi) ^{-1}$ thanks to Proposition \ref{prop11giu.0}.
\end{proof}

Now we want to obtain the Leibniz formula for the  intrinsic Lipschitz constants using an additional suitable condition on the fibers as in \cite{D22.31may}.

 \begin{theorem}[Leibniz formula]\label{sommadiLipeLip.2senzaa} 
 Let $Y$ be a metric space and $\pi: X \to Y$ be a quotient map such that  either $X=\R$ or $X=\R^s$ and each fiber is contained in a straight line. Assume also that $\pi$ satisfies \eqref{equationpinonL} and  $\phi $ and $ \psi $ are intrinsically $L$-Lipschitz sections of $\pi$ such that there is $c \geq 1$ satisfying  
 \begin{equation}\label{equationSobolev10}
d(\pi ^{-1} (y),\pi ^{-1} (z) ) \geq \frac 1 c d(f (y), \pi ^{-1} (z)), \quad \, \forall y,z \in Y,
\end{equation}
$\mbox{for } f=\phi , \psi .$  Then, denoting $\eta = \alpha \phi + \beta \psi $ the map $Y\to X$ with $\alpha , \beta \in \R-  \{ 0 \},$  we have that
\begin{equation}\label{equationLeibnitz.7}
Ils (\eta )( y) \leq   \frac c {2^{1/\lambda }}  (  Ils (\phi) ( y)+   Ils (\psi) ( y)),\quad \forall y\in Y.
\end{equation}

\end{theorem}

We need the following lemma which is true in more general setting.
 \begin{lem}\label{lem19}
Let $X$ be a normed space, $Y$ be a topological space and $\pi: X \to Y$ be a quotient map satisfies \eqref{equationpinonL}. Then
\begin{equation}\label{equation6magI}
| \alpha | d(\pi ^{-1} (y_1),\pi ^{-1} (y_2) ) =  d \left(\left(\frac 1 {\alpha ^\lambda  } \pi \right)^{-1} (y_1), \left(\frac 1 {\alpha ^\lambda  } \pi \right)^{-1} (y_2) \right), \quad \forall y_1, y_2 \in Y, \forall \alpha \in \R -\{0\}.
\end{equation}
\end{lem}

 \begin{proof}
 Fix $y_1, y_2 \in Y.$ If $d(\pi ^{-1} (y_1),\pi ^{-1} (y_2) ) =d(x_1 , x_2)$ for some $x_1\in \pi ^{-1} (y_1)$ and $x_2\in \pi ^{-1} (y_2),$ then as above (see \eqref{equation11giu.1})   we deduce that
 \begin{equation*}
 \alpha x_i \in  \pi ^{-1} (\alpha ^\lambda y_i) = \left(\frac 1 {\alpha ^\lambda  } \pi \right)^{-1} (y_i),
\end{equation*}
 for $i=1,2$ and, consequently,
 \begin{equation}\label{equation6mag1}
  \begin{aligned}
d(\pi ^{-1} (y_1),\pi ^{-1} (y_2) ) & = \|x_1-x_2\| = \frac 1 {|\alpha  |}   \|\alpha x_1-\alpha x_2 \| \\
&  \geq \frac 1 {|\alpha  |} d\left( \left( 1/ {\alpha ^\lambda  } \pi \right)^{-1} (y_1),  \left( 1/ {\alpha ^\lambda  } \pi \right)^{-1} (y_2)\right). 
\end{aligned} \end{equation} 
 On the other hand, if $d((1/\alpha ^\lambda  \pi )^{-1} (y_1), (1/\alpha ^\lambda   \pi )^{-1} (y_2) ) =d(h_1, h_2)$ for some $h_i\in (1/\alpha ^\lambda   \pi) ^{-1} (y_i)$ for $i=1,2$ then
 \begin{equation*}
\frac 1 \alpha  h_i\in  \pi ^{-1} (y_i);
\end{equation*}
indeed,
 \begin{equation*}
\pi \left(\frac 1 \alpha h_i \right) = \frac 1 {\alpha ^\lambda }\pi (h_i) =y_i, \quad \mbox{ for } i=1,2.
\end{equation*}
Hence,
\begin{equation}\label{equation6mag2}
\begin{aligned} 
d((1/\alpha ^\lambda   \pi )^{-1} (y_1), (1/ \alpha ^\lambda   \pi )^{-1} (y_2) ) & = \|h_1-h_2\|=|\alpha | \left\| \frac 1 \alpha  h_1 - \frac 1 \alpha  h_2\right\|\\
&  \geq |\alpha |  d(\pi ^{-1} (y_1),\pi ^{-1} (y_2) ). 
\end{aligned}  \end{equation}
Hence, putting together \eqref{equation6mag1} and \eqref{equation6mag2}, we get \eqref{equation6magI}.
 \end{proof}

 \begin{proof}[Proof of Proposition $\ref{sommadiLipeLip.2senzaa}$]
Notice that since $Ils (\phi)=Ils (\alpha \phi)$ for any $\alpha \ne 0$ (see Proposition \ref{prop11giu.2}), it is sufficient to prove \eqref{equationLeibnitz.7} for $\eta = \phi + \psi.$ 

We split the proof in two steps. In the first step we prove that $\eta$ is intrinsic Lipschitz and in the second one we show \eqref{equationLeibnitz.7} following a similar computation in \cite[Lemma 3.2]{KM16}. 

(1). The thesis follows from Theorem \ref{theorem11.3} noting that for $s=1$ it is trivial that each fiber is contained in a straight line. 

(2). Fix $\bar y\in Y$ and $r>0.$ Suppose we are given $\varepsilon >0.$ For any $y\in B(\bar y ,r)$ we have
\begin{equation}\label{equationSobolev10.1}
\begin{aligned}
\frac{d(\phi (\bar y), \phi (y))}{  d(\phi (\bar y), \pi ^{-1} (y)) } \leq Ils (\phi)(\bar y) + \varepsilon \qquad{and} \mbox \qquad \frac{d(\psi (\bar y), \psi (y))}{  d(\psi (\bar y), \pi ^{-1} (y)) } \leq Ils (\psi)(\bar y) + \varepsilon .
\end{aligned}
\end{equation}
On the other hand, thanks to point (1), we can find $z\in B(\bar y, r)$ so that
\begin{equation*}
\begin{aligned}
Ils (\phi +\psi )(\bar y) \leq \frac{d(\phi (\bar y)+ \psi (\bar y), \phi (z)+\psi (z))}{  d(\phi (\bar y)+\psi (\bar y), (1/2\pi )^{-1} (z)) } + \varepsilon ,
\end{aligned}
\end{equation*}
and so 
\begin{equation*}
\begin{aligned}
 Ils (\phi +\psi )(\bar y) & \leq \frac{d(\phi (\bar y), \phi (z)) + d( \psi (\bar y), \psi (z))}{  d(\phi (\bar y)+\psi (\bar y), (1/2\pi )^{-1} (z)) } +\varepsilon \\
  & \leq \frac{d(\phi (\bar y), \phi (z)) + d( \psi (\bar y), \psi (z))}{  d((1/2\pi )^{-1} (\bar y), (1/2\pi )^{-1} (z)) } +\varepsilon \\
  & = \frac{d(\phi (\bar y), \phi (z)) + d( \psi (\bar y), \psi (z))}{2^{1/\lambda } d( \pi ^{-1} (\bar y), \pi ^{-1} (z)) }+\varepsilon \\
  & \leq  \frac c {2^{1/\lambda }}   \frac{d(\phi (\bar y), \phi (z)) }{ d( \phi (\bar y), \pi ^{-1} (z)) } +   \frac c {2^{1/\lambda }}   \frac{ d( \psi (\bar y), \psi (z))}{ d( \psi (\bar y), \pi ^{-1} (z)) }+\varepsilon \\
  & \leq   \frac c {2^{1/\lambda }}  ( Ils (\phi)(\bar y) + \varepsilon ) +   \frac c {2^{1/\lambda }}   ( Ils (\psi)(\bar y) + \varepsilon )+\varepsilon
\end{aligned}
\end{equation*}
where in the first equality we used the triangle inequality, and in the second one we used the simple fact $d((1/2\pi )^{-1} (\bar y) , (1/2\pi )^{-1} (z))\leq  d(\phi (\bar y)+\psi (\bar y) , (1/2\pi )^{-1} (z)),$ noting that $\phi (\bar y)+\psi (\bar y) \in (1/2\pi )^{-1} (\bar y)$. In the first equality we used Lemma \ref{lem19} for $\alpha ^\lambda =2$ and in the last two inequalities we used \eqref{equationSobolev10} and \eqref{equationSobolev10.1}.

Hence, by the arbitrariness of $\varepsilon,$ the proof is complete.
\end{proof}
  
    \section{The intrinsic dilations vs quotient map}
    
In this section the general setting is a group $G$ together with a closed subgroup $H$ of $G$ in such a way that the quotient space  $G/H:=\{gH:g\in G\}$ naturally is a topological space for which the  map $\pi:g\mapsto gH$  is continuous, open, and surjective: it is a quotient map.

A section for the map $\pi: G\to G/H$ is just a map $\phi: G/H \to G$ such that $\phi(gH)\in gH$, since we point out the trivial identity $\pi^{-1}(gH)=gH$.  To have the notion of intrinsic Lipschitz section we need the group $G$ to be equipped with a distance which we assume left-invariant. We refer to such a $G$ as a {\em metric group}. The reader can see \cite[Section 6]{DDLD21}.

    Following the Carnot structure, we present a result where the additional hypothesis are  compatibility conditions between the distance on $G$, the quotient map and another suitable maps which corresponds to so-called intrinsic dilations in Carnot groups.
 \begin{prop}\label{carnotgroupsIDEA} 
Let $G=H_1\cdot H_2$ be a metric group and, for any fixed $\lambda >0,$ $\delta _\lambda : G \to G$ such that 
  \begin{enumerate}
  \item $\pi_{H_1} $ is $k$-Lipschitz at $1;$
    \item $d(\delta _\lambda g , \delta _\lambda q) = \lambda d(g,q)$ for any $g, q \in G;$
    \item $\delta _\lambda (\delta _ {1/\lambda } (g)) =g$ for any $g\in G;$
\item $\pi _{H_1}(\delta _\lambda g) =\delta _\lambda \pi _{H_1}(g)$ for any $g\in G.$
\end{enumerate}
Then $\delta _\lambda \circ \phi:G/H_2\to G$ is an intrinsically $L\lambda$-Lipschitz section of $\delta _ {1/\lambda } \circ \pi_{H_1}$, if $\phi :G/H_2\to G$ is so too for $\pi _{H_1}$ with Lipschitz constant $L$.
\end{prop}

 \begin{proof}
 The fact that $\phi$ is a section of  $\delta _ {1/\lambda }\pi_{H_1},$ follows from $(3)$ and
 \begin{equation*}
\pi _{H_1}(\delta _\lambda \phi (g)) = \delta_\lambda \pi _{H_1}(\phi (g)) = \delta _\lambda g, \quad \forall g\in G.
\end{equation*}
Moreover, by $(4)$ we have that for every $a,b \in G/H_2$
\begin{equation*}
d(\delta _\lambda \phi (a) , \delta _\lambda \phi (b)) =\lambda d(\phi (a), \phi (b)) \leq L \lambda d(\phi (a), \pi _{H_1}^{-1}(b)).
\end{equation*}

\end{proof}

 \begin{coroll}\label{coroll15g}
Let $G=N \rtimes H$ be a Carnot group and we consider $\delta _\lambda : G \to G$ the usual intrinsic dilations $($see, for instance, \cite{SC16}$)$. 
Then $\delta _\lambda \phi$ is an intrinsically $L\lambda$-Lipschitz section of $\delta _ {1/\lambda }\pi_{N}$, if $\phi$ is so too for $\pi _{N}$ with Lipschitz constant $L$.
\end{coroll}
    
 \begin{proof}
The only non trivial fact is $(4)$ in Proposition \ref{carnotgroupsIDEA} but it is true for the homogeneity of the polynomials that appear in  the group operation.

\end{proof}  
 
 \begin{exa}[Heisenberg groups]\label{exaHN}
The \emph{$n$-th Heisenberg group
$\mathbb{H}^n$} is the model case of Carnot groups. It is the set $\mathbb{R}^{2n+1}$ with the group
product $\cdot$ given by
\begin{displaymath}
(x_1,\ldots,x_{2n},t)\cdot
(x_1',\ldots,x_{2n}',t')=\left(x_1+x_1',\ldots,x_{2n}+x_{2n}',t+t'+\tfrac{1}{2}\sum_{i=1}^n
x_i x_{n+i}'-x_i'x_{n+i}\right)
\end{displaymath}
for
$(x_1,\ldots,x_{2n},t),(x_1',\ldots,x_{2n}',t')\in\mathbb{R}^{2n+1}$.
We equip $\mathbb{H}^n$ with the left-invariant metric
\begin{equation}\label{eq:metric}
d(p,q):=\|q^{-1}\cdot p\|,\quad p,q\in \mathbb{H}^n,
\end{equation}
where $\|(x,t)\|:=\max\{|x|,\sqrt{|t|}\}$ and $|\cdot|$ denotes
the usual Euclidean norm on $\mathbb{R}^{2n}$ and on $\R$. Here, for any fixed $\lambda >0,$ the intrinsic dilation $\delta _\lambda : \mathbb{H}^n \to \mathbb{H}^n$ is an automorphism defined as
\begin{equation*}
\delta _\lambda (x_1,\ldots,x_{2n},t) =(\lambda x_1,\ldots, \lambda x_{2n},\lambda ^2 t).
\end{equation*} Hence, it is easy to see that $\delta _\lambda (\delta _ {1/\lambda } (g)) =g$ and $\|\delta _\lambda g\| =\lambda \|g\|$ for any $g\in \HH^n.$
Moreover, if we consider the splitting $\mathbb{H}^n = H \ltimes N$ given by \begin{equation*}\label{eq:coordinates}
H=\{(x_1,\ldots,x_k,0,\ldots,0)\colon
(x_1,\ldots,x_k)\in\mathbb{R}^k\}\end{equation*}and\begin{equation*}\label{eq:coordinates2}
N=\{(0,\ldots,0,x_{k+1},\ldots,x_{2n},t)\colon
(x_{k+1},\ldots,x_{2n},t)\in\mathbb{R}^{2n+1-k}\},
\end{equation*}
then
{\small  \begin{equation*}
\begin{aligned}
\pi_N (\delta _\lambda (p)) & = \pi_N (\lambda x_1,\ldots, \lambda x_{2n},\lambda ^2 t)\\
& = \left(0,\dots , 0,\lambda x_{k+1},\ldots, \lambda x_{2n},\lambda ^2 t - \frac 1 2 \lambda ^2 \sum_{i=1}^k
x_i x_{n+i} \right).
\end{aligned}
\end{equation*} }
On the other hand,
{\small \begin{equation*}
\begin{aligned}
\delta _\lambda (\pi_N  (p) ) & = \delta _\lambda \left(0, \dots , 0, x_{k+1}, \dots , x_{2n}, t - \frac 1 2  \sum_{i=1}^k
x_i x_{n+i}\right)\\ & = \left(0,\dots , 0,\lambda x_{k+1},\ldots, \lambda x_{2n},\lambda ^2 t - \frac 1 2 \lambda ^2 \sum_{i=1}^k
x_i x_{n+i} \right),
\end{aligned}
\end{equation*} }
and so we get $$\pi_N (\delta _\lambda (p))= \delta _\lambda (\pi_N  (p) ). $$  Unfortunately, the non-linearity of $\pi$ can already be seen in this case. 

\end{exa}

 \section{The sum of intrinsic Lipschitz sections is so too in Carnot groups of step 2} 
 In this section we consider Carnot groups of step 2 as follow.
 \begin{defi}\label{def5.1.1} We say that $\G := (\R^{m+n}, \cdot  , \delta_\lambda )$ is a Carnot group of step 2 if 
there are $n$ linearly independent, skew-symmetric $m\times m$ real matrices $\mathcal{B}^{(1)}, \dots , \mathcal{B}^{(n)}$ such that  
 for all $p=(p^1,p^2)$ and $q= (q^1,q^2)\in \R^{m} \times \R^{n} $ and for all $\lambda >0$
\begin{equation}\label{1}
p\cdot q = (p^1+q^1 , p^2+q^2+ \frac{1}{2}  \langle \mathcal{B}p^1,q^1 \rangle )
\end{equation}
where $\langle \mathcal{B}p^1,q^1 \rangle := (\langle \mathcal{B}^{(1)}p^1,q^1 \rangle, \dots , \langle \mathcal{B}^{(n)}p^1,q^1 \rangle)$ 
and $\langle \cdot , \cdot \rangle$ is the inner product in $\R^m$ and 
\begin{equation*}
\delta_\lambda p  := (\lambda p^1 , \lambda^2 p^2).
\end{equation*}
\end{defi}
When $H$ is one dimensional, without loss of generality (see \cite[Chapter 3.4]{BLU}), we consider the following splitting
\begin{equation*}\label{5.2.0}
H:=\{ (x_1,0\dots , 0)\,:\, x_1 \in \R \}, \qquad N:=\{ (0,x_2,\dots , x_{m+n})\,:\, x_i \in \R \}.
\end{equation*}

 Unfortunately, in the context of Carnot groups of step two  $\pi_N$ is non linear (see the model case as in Example \ref{exaHN}). However, with the additional hypothesis which is a sort of a compatibility condition between two Lipschitz sections it is possible to deduce the following proposition in Carnot groups of step 2.
   \begin{prop}\label{carnotgroupsIDEA.12} 
Let $G=N \rtimes H$ be a Carnot group of step 2 with $H$ 1-dimensional and $\phi , \psi : G/H \to G$ be intrinsically Lipschitz sections of $\pi _{G/H}$ such that if we put 
\begin{equation*}
\begin{aligned}
 &  \mathcal{A}_{1,\ell}:= \langle \mathcal{B}^{(\ell)}p^1,q^1  \rangle  - \langle \mathcal{B}^{(\ell)}(0,p_2 + q_2,\dots , p_m +q_{m}) , (p_1+q_1,0\dots , 0) \rangle ,\\
 &   \mathcal{A}_{2,\ell}:=  \langle \mathcal{B}^{(\ell)}( 0, p_{2},\dots , p_{m}),( 0, q_{2},\dots , q_{m})  \rangle +  \langle \mathcal{B}^{(\ell)}( 0, p_{2},\dots , p_{m}),(p_1, 0)  \rangle \\
 & \quad +  \langle \mathcal{B}^{(\ell)}( 0, q_{2},\dots , q_{m}),( q_1, 0)  \rangle ,
 \end{aligned} 
\end{equation*} $\forall \ell =1,\dots , n$  where $p_i=\phi _i(a)$ and $q_i=\psi _i(b)$ for any $i=1,\dots , m+n,$ we want that
\begin{equation}\label{equationCOMPATIBILITA}
\begin{aligned}
 &  \mathcal{A}_{1,\ell} = \mathcal{A}_{2,\ell}, \quad \forall \ell =1,\dots , n.
 \end{aligned}
\end{equation} 

Then the sum of $\phi$ and $\psi$ is intrinsic Lipschitz.
\end{prop}
 
  \begin{proof}
We want to apply Corollary 4.6 in \cite{D22.31may}. The only hypothesis that we need is the linearity of $\pi$ which we don't have in general. However, it is sufficient to show that
\begin{equation*}
\pi ( \phi (y_1) + \psi (y_2)) = \pi(\phi (y_1)) + \pi(\psi (y_2)), \quad \forall y_1, y_2 \in G/H,
\end{equation*} where $\phi $ and $\psi$ satisfy \eqref{equationCOMPATIBILITA}. 
For simplicity, we identify $G/H$ with $N,$ and $\pi_{G/H}$ with $\pi_N$ and we just write $\pi.$ We prove
\begin{equation}\label{equationIMPO12giugno}
\begin{aligned}
\pi(p\cdot q)= \pi (p) \cdot \pi (q), 
\end{aligned}
\end{equation}
for $p=(p^1, p^2)=(p_1,\dots , p_{m+n})$ and $q=(q^1, q^2)=(q_1,\dots , q_{m+n})$ when $ \mathcal{A}_{1,\ell} = \mathcal{A}_{2,\ell}.$  By an easy computation, we get
\begin{equation*}
\begin{aligned}
\pi(p)= \left( 0, p_{2},\dots , p_{m}, [\pi(p)]_{m+1}, \dots , [\pi(p)]_{m+n }  \right)
\end{aligned}
\end{equation*}
where
\begin{equation*}
\begin{aligned}
[\pi(p)]_{m+\ell } & = p_{m+\ell} +\frac 1 2  \langle \mathcal{B}^{(\ell)}(0,p_2,\dots , p_m ) , (p_1,0\dots , 0) \rangle ,
\end{aligned}
\end{equation*}
for $\ell = 1,\dots , n.$ Consequently,
\begin{equation*}
\pi(p)\cdot \pi(q)= (0, p_{2}+q_2,\dots , p_{m}+q_{m}, [\pi(p)\cdot \pi(q)]_{m+1} , \dots , [\pi(p)\cdot \pi(q)]_{m+n } ),
\end{equation*}
where
\begin{equation*}
\begin{aligned}
[\pi(p)\cdot \pi(q)]_{m+\ell } & =  \frac 1 2  \langle \mathcal{B}^{(\ell)}( 0, p_{2},\dots , p_{m}),( 0, q_{2},\dots , q_{m})  \rangle  + [\pi(p)]_{m+\ell } + [\pi(q)]_{m+\ell }.\\
\end{aligned}
\end{equation*}
On the other hand,
 \begin{equation*}
\begin{aligned}
\pi(p\cdot q)= \left( 0, p_{2}+q_2,\dots , p_{m}+q_{m}, [\pi(p\cdot q)]_{m+1}, \dots , [\pi(p\cdot q)]_{m+n }  \right),
\end{aligned}
\end{equation*}
where
\begin{equation*}
\begin{aligned}
[\pi(p\cdot q)]_{m+\ell } & = p_{m+\ell} +q_{m+\ell} + \frac 1 2  \langle \mathcal{B}^{(\ell)}p^1,q^1  \rangle \\
& \quad -\frac 1 2  \langle \mathcal{B}^{(\ell)}(0,p_2 + q_2,\dots , p_m +q_{m}) , (p_1+q_1,0\dots , 0) \rangle ,
\end{aligned}
\end{equation*}
for $\ell = 1,\dots , n.$ Hence we get the thesis using the hypothesis $ \mathcal{A}_{1,\ell} = \mathcal{A}_{2,\ell}.$
\end{proof}  
 
  \begin{exa}[Heisenberg groups] Let $G=\HH^n = N \rtimes H$ as in Example \ref{exaHN}. By an easy computation, we get an explicit formula of  \eqref{equationCOMPATIBILITA} (here $\ell =1$), i.e., either 
  \begin{equation*}
\phi_1 \equiv 0
\end{equation*}
or 
  \begin{equation*}
\psi _{n+1} \equiv 0.
\end{equation*}
   \end{exa}
 
 \bibliographystyle{alpha}
\bibliography{DDLD}

\newcommand{\etalchar}[1]{$^{#1}$}
\begin{thebibliography}{FSSC03b}

\bibitem[ABB19]{ABB}
Andrei Agrachev, Davide Barilari, and Ugo Boscain.
\newblock A comprehensive introduction to sub-{R}iemannian geometry.
\newblock {\em Cambridge Studies in Advanced Mathematics, Cambridge Univ.
  Press}, 181:762, 2019.

\bibitem[AGS14a]{AGS08.2}
L.~Ambrosio, N.~Gigli, and G.~Savar\'e.
\newblock Calculus and heat flow in metric measure spaces and applications to
  spaces with {R}icci bounds from below.
\newblock {\em Invent. Math.}, 195(2):289--391, 2014.

\bibitem[AGS14b]{AGS08.3}
L.~Ambrosio, N.~Gigli, and G.~Savar\'e.
\newblock Metric measure spaces with {R}iemannian {r}icci curvature bounded
  from below.
\newblock {\em Duke Math. J.}, 163(7):1405--1490, 2014.

\bibitem[AK00]{AmbrosioKirchheimRect}
Luigi Ambrosio and Bernd Kirchheim.
\newblock Rectifiable sets in metric and {B}anach spaces.
\newblock {\em Math. Ann.}, 318(3):527--555, 2000.

\bibitem[BJL{\etalchar{+}}99]{MR1736929}
S.~Bates, W.~B. Johnson, J.~Lindenstrauss, D.~Preiss, and G.~Schechtman.
\newblock Affine approximation of {L}ipschitz functions and nonlinear
  quotients.
\newblock {\em Geom. Funct. Anal.}, 9(6):1092--1127, 1999.

\bibitem[BLU07]{BLU}
A.~Bonfiglioli, E.~Lanconelli, and F.~Uguzzoni.
\newblock {\em Stratified {L}ie groups and potential theory for their
  sub-{L}aplacians}.
\newblock Springer Monographs in Mathematics. Springer, Berlin, 2007.

\bibitem[CDPT07]{CDPT}
Luca Capogna, Donatella Danielli, Scott~D. Pauls, and Jeremy~T. Tyson.
\newblock {\em An introduction to the {H}eisenberg group and the
  sub-{R}iemannian isoperimetric problem}, volume 259 of {\em Progress in
  Mathematics}.
\newblock Birkh\"{a}user Verlag, Basel, 2007.

\bibitem[Che99]{C99}
J.~Cheeger.
\newblock Differentiability of {L}ipschitz functions on metric measure spaces.
\newblock {\em Geom. Funct. Anal.}, 9(8):428--517, 1999.

\bibitem[DD20]{DD20}
Daniela Di~Donato.
\newblock Intrinsic {L}ipschitz graphs in {C}arnot groups of step 2.
\newblock {\em Ann. Acad. Sci. Fenn. Math}, pages 1--51, 2020.

\bibitem[DD22a]{D22.31may}
Daniela Di~Donato.
\newblock Intrinsic {C}heeger energy for the intrinsically {L}ipschitz
  constants.
\newblock {\em preprint}, 2022.

\bibitem[DD22b]{D22.1}
Daniela Di~Donato.
\newblock Intrinsically {H}\"older sections in metric spaces.
\newblock {\em preprint}, 2022.

\bibitem[DDLD22]{DDLD21}
Daniela Di~Donato and Enrico Le~Donne.
\newblock Intrinsically {L}ipschitz sections and applications to metric groups.
\newblock {\em preprint}, 2022.

\bibitem[FMS14]{FMS14}
Bruno Franchi, Marco Marchi, and Raul Serapioni.
\newblock Differentiability and approximate differentiability for intrinsic
  lipschitz functions in carnot groups and a rademarcher theorem.
\newblock {\em Anal. Geom. Metr. Spaces}, 2(3):258--281, 2014.

\bibitem[FS16]{FS16}
Bruno Franchi and Raul~Paolo Serapioni.
\newblock Intrinsic {L}ipschitz graphs within {C}arnot groups.
\newblock {\em J. Geom. Anal.}, 26(3):1946--1994, 2016.

\bibitem[FSSC01]{FSSC}
B.~Franchi, R.~Serapioni, and F.~Serra~Cassano.
\newblock Rectifiability and perimeter in the {H}eisenberg group.
\newblock {\em Math. Ann.}, 321(3):479--531, 2001.

\bibitem[FSSC03a]{MR2032504}
B.~Franchi, R.~Serapioni, and F.~Serra~Cassano.
\newblock Regular hypersurfaces, intrinsic perimeter and implicit function
  theorem in {C}arnot groups.
\newblock {\em Comm. Anal. Geom.}, 11(5):909--944, 2003.

\bibitem[FSSC03b]{FSSC03}
Bruno Franchi, Raul Serapioni, and Francesco Serra~Cassano.
\newblock On the structure of finite perimeter sets in step 2 {C}arnot groups.
\newblock {\em The Journal of Geometric Analysis}, 13(3):421--466, 2003.

\bibitem[Ha96]{H96}
P.~Haj\l~asz.
\newblock Sobolev spaces on an arbitrary metric space.
\newblock {\em Potential Analysis, 5}, pages 403--415, 1996.

\bibitem[JPNJ15]{HKST15}
Heinonen J., Koskela P., Shanmugalingam N., and Tyson. J.
\newblock Sobolev spaces on metric measure spaces. { A}n approach based on
  upper gradients.
\newblock {\em New Mathematical Monographs, 27. Cambridge: Cambridge University
  Press}, 2015.

\bibitem[Kei04]{K04}
S.~Keith.
\newblock A differentiable structure for metric measure spaces.
\newblock {\em Adv. Math. 183}, pages 271--315, 2004.

\bibitem[KM16]{KM16}
B.~Kleiner and J.M. Mackay.
\newblock Differentiable structures on metric measure spaces: a primer.
\newblock {\em Ann. Sc. Norm. Super. Pisa Cl. Sci. (5) Vol. XVI}, pages 41--64,
  2016.

\bibitem[LV09]{MR2480619}
John Lott and C\'{e}dric Villani.
\newblock Ricci curvature for metric-measure spaces via optimal transport.
\newblock {\em Ann. of Math. (2)}, 169(3):903--991, 2009.

\bibitem[Pan89]{Pansu}
Pierre Pansu.
\newblock M\'{e}triques de {C}arnot-{C}arath\'{e}odory et quasiisom\'{e}tries
  des espaces sym\'{e}triques de rang un.
\newblock {\em Ann. of Math. (2)}, 129(1):1--60, 1989.

\bibitem[SC16]{SC16}
Francesco Serra~Cassano.
\newblock Some topics of geometric measure theory in {C}arnot groups.
\newblock In {\em Geometry, analysis and dynamics on sub-{R}iemannian
  manifolds. {V}ol. 1}, EMS Ser. Lect. Math., pages 1--121. Eur. Math. Soc.,
  Z\"urich, 2016.

\bibitem[Sha00]{S00}
N.~Shanmugalingam.
\newblock Newtonian spaces: an extension of {S}obolev spaces to metric measure
  spaces.
\newblock {\em Rev. Mat. Iberoam., 16}, pages 243--279, 2000.

\bibitem[Stu06]{S06}
K.T. Sturm.
\newblock On the geometry of metric measure spaces. {I}, {II}.
\newblock {\em Acta Math.}, 196(1):65-- 131 and 133--177, 2006.

\bibitem[Vil09]{MR2459454}
C\'{e}dric Villani.
\newblock {\em Optimal transport, Old and new}, volume 338 of {\em Grundlehren
  der Mathematischen Wissenschaften [Fundamental Principles of Mathematical
  Sciences]}.
\newblock Springer-Verlag, Berlin, 2009.

\bibitem[Vit20]{Vittone20}
Davide Vittone.
\newblock Lipschitz graphs and currents in {H}eisenberg groups.
\newblock 2020.
\newblock Preprint, available at \url{https://arxiv.org/abs/2007.14286}.

\bibitem[VN88]{Berestovski}
Berestovskii Valerii~Nikolaevich.
\newblock Homogeneous manifolds with intrinsic metric.
\newblock {\em Sib Math J}, I(29):887--897, 1988.

\end{thebibliography}

\end{document}